\documentclass[a4paper]{amsart}

\usepackage{amsthm,amsfonts,amsmath, amssymb}

\let\oldmarginpar\marginpar
\renewcommand{\marginpar}[1]{\oldmarginpar{\small\textit{{#1}}}}

\usepackage{setspace}
\usepackage{graphicx,float,enumerate,subfigure,tikz}
\usepackage[ruled,boxed,commentsnumbered,norelsize]{algorithm2e}
\usepackage{verbatim}
\usepackage{url}   
\usepackage[capitalise]{cleveref}

\usepackage[sort&compress,comma]{natbib}
 \usepackage{enumitem}
\setlist[enumerate,1]{label=(\roman*),font=\normalfont}

\newtheorem{theorem}{Theorem}
\newtheorem{corollary}[theorem]{Corollary}
\newtheorem{lemma}[theorem]{Lemma}
\newtheorem{proposition}[theorem]{Proposition}

\theoremstyle{definition}

\theoremstyle{remark}

\crefname{remark}{Remark}{Remarks}
\crefname{claim}{Claim}{Claims}

\theoremstyle{definition}

\usepackage{etoolbox}
\AtEndEnvironment{proof}{\setcounter{claim}{0}}

\theoremstyle{remark}

\usepackage{etoolbox}
\AtEndEnvironment{proof}{\setcounter{case}{0}}

\crefname{rmk}{Remark}{Remarks} 
\crefname{problem}{Problem}{Problems}
\crefname{conjecture}{Conjecture}{Conjectures}

\usepackage{mathtools}

\DeclareMathOperator{\conv}{conv}

\DeclareMathOperator{\lk}{lk}

\newcommand{\R}{\mathbb{R}}

\renewcommand{\ge}{\geqslant} 
\renewcommand{\le}{\leqslant}

\date{\today}
\title[A new proof of Balinski's theorem]{A new proof of Balinski's theorem on the connectivity of polytopes} 
 
\usepackage{amsaddr}

\author{Guillermo Pineda-Villavicencio}
\address{Centre for Informatics and Applied Optimisation, Federation University,  Australia\\School of Information Technology, Deakin University, Geelong,  Australia} 
\email{\texttt{work@guillermo.com.au}}

 
\begin{document}
\begin{abstract} 
 \cite{Bal61} proved that the graph of a $d$-dimensional convex polytope is $d$-connected. We provide a new proof of this result. Our proof provides details on the nature of a separating set  with exactly $d$ vertices; some of which appear to be new.     \end{abstract}

\maketitle  
 
\section{Introduction} 
A (convex) polytope is the convex hull of a finite set $X$ of points in $\R^{d}$; the \textit{convex hull} of $X$ is  the smallest convex set containing $X$.  The \textit{dimension} of a polytope in $\R^{d}$ is one less than the maximum number of affinely independent points in the polytope; a set of points $\vec p_{1},\ldots, \vec p_{k}$ in $\R^{d}$ is {\it affinely independent} if  the $k-1$ vectors $\vec p_{1}-\vec p_{k},\ldots, \vec p_{k-1}-\vec p_{k}$ are linearly independent.  A polytope of dimension $d$ is referred to as a \textit{$d$-polytope}.

A polytope  is structured around other polytopes, its faces. A {\it face} of a polytope $P$ in $\R^{d}$ is $P$ itself, or  the intersection of $P$ with a hyperplane in $\R^{d}$ that contains $P$ in one of its closed halfspaces.  A face of dimension 0, 1, and $d-1$ in a $d$-polytope is a \textit{vertex}, an {\it edge}, and a {\it facet}, respectively. The set of vertices and edges of a polytope or a graph are denoted by $V$ and $E$, respectively. The \textit{graph} $G(P)$ of a polytope  $P$ is the abstract graph with vertex set $V(P)$ and edge set $E(P)$.

 A graph with at least $d+1$ vertices is \textit{$d$-connected} if removing any $d-1$ vertices leaves a connected subgraph.  \cite{Bal61} showed that the graph of a $d$-polytope is $d$-connected. His proof considers a hyperplane in $\R^{d}$ passing through a set of $d-1$ vertices of a $d$-polytope, and so do the proofs of \citet[Thm.~11.3.2]{Gru03}, \citet[Thm.~3.14]{Zie95}, and \citet[Thm.~15.6]{Bro83}. Such proofs yield a geometric structure of separators in the graph of the polytope (\cref{lem:separator-hyperplane}). A set $X$ of vertices in a graph $G$ \textit{separates} two vertices $x,y$ if every  path in $G$ between $x$ and $y$ contains an element of $X$, and $x,y\notin X$. And  $X$ \textit{separates} $G$ if it separates two vertices of $G$. A separating set of vertices is a \textit{separator}  and a separator of cardinality $r$ is an \textit{$r$-separator}.
  
\begin{lemma}
\label{lem:separator-hyperplane}
Let $P$ be a $d$-polytope in $\R^{d}$ and let $H$ be a hyperplane in $\R^{d}$.  If $X$ is a proper subset of $H\cap V(P)$, then removing $X$ does not disconnect $G(P)$. In particular, a separator of $G(P)$ with exactly  $d$ vertices must form an affinely independent set in $\R^{d}$.
\end{lemma}

Other proofs with a geometric flavour were given by  \cite{BroMax89}  and \cite{Bar95}.  Our proof has a more combinatorial nature, relying on certain polytopal complexes in a polytope. Another combinatorial proof, based on a different idea,  can be found in \cite{Bar73a}.

The \textit{boundary complex} of a polytope $P$ is the set of faces of $P$ other than $P$ itself. And the \textit{link} of a vertex $x$ in $P$, denoted $\lk(x)$, is the set of faces of $P$ that do not contain $x$ but lie in a facet of $P$ that contains $x$ (\cref{fig:link-polytope}(b)).  We require a result from \citet{Zie95}. 
  
\begin{proposition}[{\citealt[Ex.~8.6]{Zie95}}]\label{prop:link-polytope} Let $P$ be a $d$-polytope. Then the link of a vertex in $P$ is combinatorially isomorphic to the boundary complex of a $(d-1)$-polytope. In particular, for each $d\ge 3$, the graph of the link of a vertex is isomorphic to the graph of a $(d-1)$-polytope.
\end{proposition}

We proved \cref{prop:link-polytope} in \citet[Prop.~12]{BuiPinUgo18} and exemplified it in \cref{fig:link-polytope}. In this paper, we prove the following. The  part about links appears to be new.

\begin{figure} 
\includegraphics{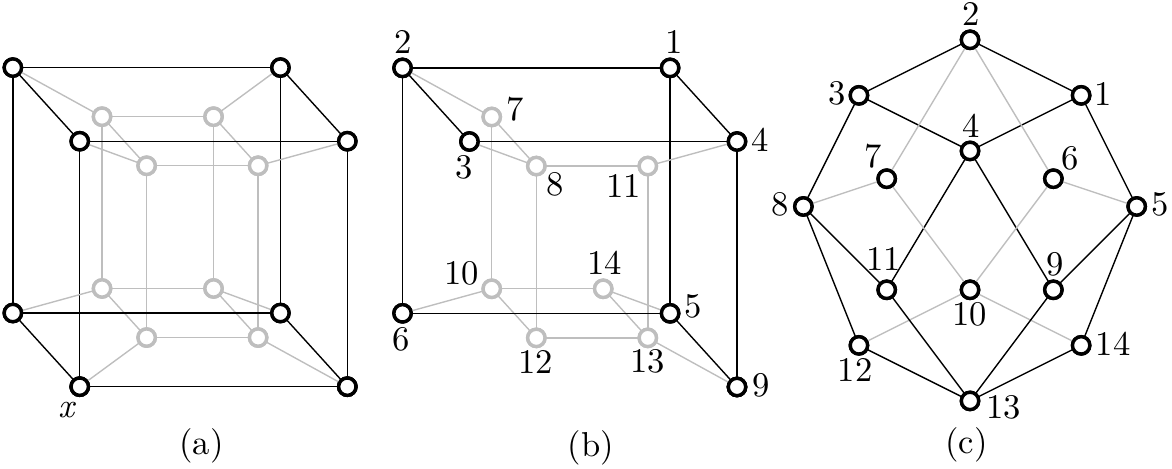}
\caption{The link of a vertex in the four-dimensional cube, the convex hull of the $2^{4}$ vectors $(\pm 1,\pm 1,\pm 1,\pm 1)$ in $\R^{4}$.  (a) The four-dimensional cube with a vertex $x$ highlighted. (b)  The link of the vertex $x$ in the cube.  (c) The link of the vertex $x$ as the boundary complex of the rhombic dodecahedron (\cref{prop:link-polytope}).} \label{fig:link-polytope}   
\end{figure}

\begin{theorem}
\label{thm:Balinski} For $d\ge 1$, the graph of $d$-polytope $P$ is $d$-connected. Besides, for each $d\ge 3$, each vertex $x$ in a $d$-separator $X$ of $G(P)$ lies in the link of every other vertex of $X$, and the set $X\setminus \left\{x\right\}$ is a separator of the link of $x$.  
\end{theorem}

As a corollary, we get a known result on $d$-separators in simplicial $d$-polytopes \citep[p.~509]{GooORo17-3rd}; see \cref{cor:link-simplicial}. A polytope is \textit{simplicial} if all its facets are simplices, and  a {\it $d$-simplex} is a $d$-polytope whose $d+1$ vertices form an affinely independent set in $\mathbb R^d$. An \textit{empty $(d-1)$-simplex} in a $d$-polytope $P$ is a set of $d$ vertices of $P$ that does not form a face of $P$ but every proper subset does. An empty $(d-1)$-simplex is also called a \textit{missing $(d-1)$-simplex}. 

\begin{corollary} Let $P$ be a simplicial $d$-polytope with $d\ge 2$. A $d$-separator of $G(P)$ forms an empty $(d-1)$-simplex of $P$.
\label{cor:link-simplicial}  
\end{corollary}

We remark that the paragraph after Balinski's theorem in \citet[p.~509]{GooORo17-3rd}  is meant to concern only simplicial $d$-polytopes, and not $d$-polytopes in general. While it is true that a  $d$-separator  of the graph of a $d$-polytope must form an affinely independent set in $\R^{d}$, it is not true that it must form an empty simplex. Take, for instance, the neighbours of a vertex in a $d$-dimensional cube (\cref{fig:link-polytope}(a)).	 
 
We follow \cite{Die05} for the graph theoretical terminology that we have not defined. 

\section{Proofs of \cref{thm:Balinski} and \cref{cor:link-simplicial}}

  A  path between vertices $x$ and $y$ in a graph is an {\it $x-y$ path}, and two $x-y$ paths  are {\it independent}\index{independent} if they share no inner vertex.  For a path $L:=x_{0}\ldots x_{n}$ and for $0\le i\le j\le n$, we write $x_{i}Lx_{j}$  to denote the subpath $x_{i}\ldots x_{j}$. We require a theorem of \cite{Whi32} and one of \cite{Menger1927}.
 
 \begin{theorem}[{\citealt{Whi32}}]\label{thm:Whitney}  Let $G$ be a  graph with at least one pair of nonadjacent vertices. Then there is a minimum separator of $G$ disconnecting two nonadjacent vertices. 
\end{theorem}

 \begin{theorem}[{\citealt{Menger1927}}]\label{thm:Menger} Let $G$ be a  graph, and let $x$ and $y$ be two nonadjacent vertices. Then the minimum number of vertices separating $x$ from $y$ in $G$ equals the maximum number of independent $x-y$ paths in $G$. 
\end{theorem}

\begin{proof}[Proof of \cref{thm:Balinski}]

Let $P$ be a $d$-polytope and let $G$ be its graph. Then $G$ has at least $d+1$ vertices. If $G$ is a complete graph, there is nothing to prove, and suppose otherwise. In this case, $G$ has at least one pair of nonadjacent vertices.  For $d=2$, $G$ is $d$-connected. And so induct on $d$, assuming that $d\ge 3$ and that the theorem is true for $d-1$. Let  $X$ be a separator in $G$ of minimum cardinality, and let $y$ and $z$ be vertices separated by $X$. Then $y,z\notin X$. According to Whitney's theorem (\cref{thm:Whitney}), there is a minimum separator of $G$ disconnecting two nonadjacent vertices. Hence we may assume that $y$ and $z$ are nonadjacent, and by Menger's theorem (\cref{thm:Menger}), that there are $|X|$ independent $y-z$ paths in $G$, each containing precisely one vertex from $X$. Let $L$ be one such $y-z$ paths and let $x$ be the vertex in $X\cap V(L)$; say that $L=u_{1}\ldots u_{m}$ such that $y=u_{1}$, $u_{j}=x$, and $u_{m}=z$.  

The graph $G_{x}$ of the link of $x$ in $P$ is isomorphic to the graph of a $(d-1)$-polytope (\cref{prop:link-polytope}), and by the induction hypothesis it is $(d-1)$-connected. The neighbours of $x$ are all part of  $\lk(x)$, and so $u_{j-1},u_{j+1}\in G_{x}$.  Again, from Menger's theorem  follows the existence of at least $d-1$ independent $u_{j-1}-u_{j+1}$ paths in $G_{x}$.  We must have that $X\setminus \left\{x\right\}$ separates  $u_{j-1}$ from $u_{j+1}$ in $G_{x}$, since $X$ separates $y$ from $z$.  Hence $|X\setminus \left\{x\right\}|\ge d-1$, which establishes that $G$ is $d$-connected.    

Finally, let $d\ge 3$ and suppose $X$ is a $d$-separator of $G$. As stated above, the set $X\setminus \left\{x\right\}$, of cardinality $d-1$, separates $G_{x}$, implying that $X\setminus \left\{x\right\} \subseteq V(G_{x})$. The aforementioned path $L$ was arbitrary among the $y-z$ paths separated by $X$, and each such path contains a unique vertex of $X$. It follows that every vertex in $X$ is in the link of every other vertex of $X$, which  concludes the proof of the theorem.      
\end{proof}
	
\begin{proof}[Proof of \cref{cor:link-simplicial}] Let $P$ be a simplicial $d$-polytope and let $G$ be its graph.  
Suppose that $X$ is a $d$-separator of $G$, that  $x$ is a vertex of $X$,  and that  $G_{x}$ is the graph of the link of $x$ in $P$.  

A simplicial 2-polytope is a polygon and a 2-separator in it satisfies the corollary. So assume that $d\ge 3$. From \cref{thm:Balinski}, it follows that every vertex in $X$ is in the link of every other vertex of $X$, and that $X\setminus\left\{x\right\}$ is a $(d-1)$-separator of $G_{x}$. Consequently, the subgraph $G[X]$ of $G$ induced by $X$ is a complete graph, as the set of neighbours of each vertex in $P$ coincides with the vertex set of the link of the vertex.    

If $d=3$, then, from  $G[X]$ being a complete graph, it follows that it is an empty $2$-simplex. And so an inductive argument on $d$ can start. Assume that $d\ge 4$. From the definition of a link and \cref{prop:link-polytope}, we obtain that $\lk(x)$ is combinatorially isomorphic to the boundary complex of a simplicial $(d-1)$-polytope. 
 
By the induction hypothesis on $\lk(x)$, every proper subset of $X\setminus\left\{x\right\}$ forms a face $F$ of $\lk(x)$. And from the definition of $\lk(x)$, that face $F$ lies in a facet of $P$ containing $x$, a $(d-1)$-simplex containing $x$. As a consequence, if $F$ is a face of dimension $k$, then the set $\conv(F\cup \left\{x\right\})$ is a face of $P$ of dimension $k+1$. Since the vertex $x$ of $X$ was taken arbitrarily, the corollary ensues.
\end{proof}


\end{document}